\documentclass[12pt,reqno]{amsart}
\usepackage {amssymb,amsfonts}
\usepackage {amsmath}
\usepackage{amsthm}
\usepackage{mathtools}
\usepackage[old]{old-arrows} 
\usepackage{stmaryrd}

\usepackage[utf8]{inputenc}
\usepackage[english]{babel}

\usepackage{tikz}
\usepackage{tikz-cd}
\usetikzlibrary{babel}

\usepackage{graphicx}
\usepackage[alphabetic]{amsrefs}
\usepackage[bookmarks=false,hypertexnames=false,colorlinks, citecolor=blue]{hyperref}
\usepackage{bookmark}
\usepackage{verbatim,color,geometry}
\usepackage[colorinlistoftodos, textsize=small]{todonotes}
\makeatletter 
\@mparswitchfalse%
\makeatother
\reversemarginpar

\usepackage[shortlabels]{enumitem}

\newtheorem{prop}{Proposition}[section]
\newtheorem{thm}[prop]{Theorem}
\newtheorem{defn-thm}[prop]{Theorem-Definition}
\newtheorem{defn-lem}[prop]{Lemma-Definition}
\newtheorem{cor}[prop]{Corollary}

\newtheorem{lem}[prop]{Lemma}
\newtheorem*{thm*}{Theorem}
\newtheorem*{cor*}{Corollary}

\theoremstyle{definition}

\newtheorem{defn}[prop]{Definition}
\newtheorem{example}[prop]{Example}
\newtheorem{rem}[prop]{Remark}

\newtheorem*{claim*}{Claim}
\newtheorem*{rem*}{Remark}

\theoremstyle{remark}

\newcommand{\su}{\subset}

\newcommand{\wt}{\widetilde}

\newcommand{\ov}{\overline}

\newcommand{\isom}{\cong}

\newcommand{\Q}{\mathbb{Q}}

\newcommand{\PP}{\mathbb{P}}

\newcommand{\F}{\mathbb{F}}

\newcommand{\OO}{\mathcal{O}}

\newcommand{\pp}{\mathfrak{p}}
\newcommand{\fq}{\mathfrak{q}}

\newcommand{\de}{\delta}
\newcommand{\De}{\Delta}

\newcommand{\ka}{\kappa}

\newcommand{\lam}{\lambda}

\begin{document}

\title{Turning non-smooth points into rational points}

\author{Cesar Hilario}
\address{Departamento de Ciencias, Pontificia Universidad Cat\'olica del Per\'u, Av. Universitaria 1801, San Miguel 15088, Lima, Peru}
\email{cesar.hilario.pm@gmail.com}



\dedicatory{October 29, 2026}

\begin{abstract}
In a recent paper, the author and Stöhr established a bound on the number of iterated Frobenius pullbacks needed to transform a non-smooth purely inseparable point on a regular geometrically integral curve into a rational point.
In this paper we improve this result, 
by establishing a new bound that is sharp in every characteristic $p>0$.

\end{abstract}

\maketitle

\setcounter{tocdepth}{1}

\section{Introduction}

This paper is concerned with the behavior of non-smooth points on regular curves under iterated Frobenius pullbacks.

\subsection{Regularity and smoothness}
A point $x$ on a curve%
\footnote{All our curves are assumed to be geometrically integral and projective.}
$C$ over a field $K$ is called \emph{regular} if its local ring $\OO_{C,x}$ is regular.
It is called \emph{smooth} if it is regular and, in addition, 
the preimages of $x$ under the projection map $C_{\ov K} \to C$ are regular points on the extended curve $C_{\ov K} = C \otimes_K \ov K$ over the algebraic closure $\ov K$;
see \cite[Chapt.~0, 22.5.8]{EGA4a}, \cite{Fal78}, or \cite[p.\,141]{Liu02}.

If the base field $K$ is perfect (e.g., $\mathrm{char}(K)=0$ or $K = \ov K$),
regularity and smoothness are equivalent notions.
If $K$ is imperfect, however, smoothness is stronger than regularity.
A typical example of an imperfect field is the function field of an algebraic variety over a field of positive characteristic, such as $\F_q(t)$, where $t$ is transcendental over the finite field $\F_q$.
In this context, regular curves that are non-smooth arise, making this phenomenon a distinctive feature of geometry in characteristic $p>0$.
For instance, \emph{quasi-elliptic curves} (i.e., regular non-smooth curves of genus 1) occur as generic fibres of generically non-smooth fibrations on surfaces in characteristics $p=2,3$, which is the reason why the classification of algebraic surfaces is subtler in these characteristics than in other characteristics;
see \cite{BM77,BM76} or \cite[Chapter~4]{CDL23}.

\subsection{Non-smooth points and Frobenius pullbacks}
Let $C$ be a regular curve over a field $K$ of characteristic $p>0$. 
We consider the iterated Frobenius pullbacks $C^{(p^n)}$ of $C$, which fit into an infinite sequence of relative Frobenius morphisms
\[ C^{(p^0)} = C\to C^{(p^1)} \to C^{(p^2)} \to C^{(p^3)} \to \cdots. \]
Passing to their normalizations $C_n$ we obtain a sequence of regular curves over $K$
\[ C_0 = C \to C_1 \to C_2 \to C_3 \to \cdots. \]
Let $x$ be a non-smooth point on $C$. 
The non-smoothness of $x$ means that its geometric singularity degree $\de(x) \geq 0$ is positive (see Definition~\ref{2025_05_18_21:20}).
One knows that Frobenius pullbacks turn the non-smooth point $x$ into a smooth point, that is, the image point $x_n \in C_n$ is smooth for $n$ large enough.
It was proved in \cite{HiSt22} that the non-smooth point $x$ is actually turned into a $K$-rational point, if in addition $x$ is assumed to be \emph{purely inseparable}, 
\footnote{In \cite{HiSt22} one uses the term \emph{non-decomposed}; see Definition~\ref{2025_05_18_23:00}.}
i.e., $x \in C(K^{p^{-\infty}})$, where $K^{p^{-\infty}}$ is the maximal purely inseparable extension of $K$.
Purely inseparable points, also known in the literature as \emph{perfect}, have been studied in recent years by many authors \cite{Ghi10,Ros15,Ros20,Yuan21,BL22,Amb23}.


\begin{defn}
Let $C|K$ be a regular curve equipped with a purely inseparable point $x$. We call $(C|K,x)$ a \emph{pair}.
\end{defn}

Given a pair $(C|K,x)$ with $x$ a non-smooth point,
the bound $m$ such that $x_n \in C_n(K)$ for all $n\geq m$ depends on the singularity degree $d=\de(x)>0$ and the characteristic $p>0$, but not on the curve $C$ nor on the field $K$ (see \cite[Theorem~1.1]{HiSt22}).
In characteristic $p=2$ the bound is sharp (see \cite[Proposition~2.25]{HiSt22}), that is, for each integer $d>0$ the bound $m$ is optimal in the sense that there is a pair $(C|K,x)$ with $\de(x) = d$ such that $x_{m-1} \notin C_{m-1}(K)$.

In characteristic $p>2$ there is less flexibility, because on the one hand the singularity degrees $\de(x)$ are multiples of $\frac{p-1}2$ (see \cite[Proposition~2.5]{HiSt22}), while on the other hand
not every such multiple can occur as the singularity degree $\de(x)$ of a pair $(C|K,x)$.
We show (see Proposition~\ref{2023_07_14_13:25}) that a positive integer $d$ is \textit{admissible}, i.e., there exists a pair $(C|K,x)$ with $\de(x) = \frac{p-1}2 {\cdot} d$, 
if and only if the congruence relation $d \not \equiv -1 \pmod p$ is satisfied.
The goal of this paper is to determine a sharp bound for every admissible $d>0$.

\begin{thm}\label{2025_05_22_11:10}
Let $(C|K,x)$ be a pair with $\de(x) = \frac{p-1}2 {\cdot}d > 0$. 
Then $x_n \in C_n$ is a $K$-rational point for all $n \geq \lam_p \big( d \big)$, where $\lam_p(d)$ is the integer determined in Theorem~\ref{2022_07_24_00:25}.

Furthermore the bound $\lam_p(d)$ is sharp for $d>0$, that is, 
there exists a pair $(C'|K',x')$ with $\de(x') = \frac{p-1}2 {\cdot}d$,
such that $x'_n \in C'_n$ is not a $K'$-rational point for $n=\lam_p(d)-1$.
\end{thm}



One application of results like Theorem~\ref{2025_05_22_11:10} is a method to achieve explicit descriptions of regular non-smooth curves $C|K$  (see \cite[p.\,281]{HiSt22}). 
To explain how the method works, suppose $C$ is geometrically rational, i.e., the base change $C_{\ov K}$ has geometric genus $\ov g = h^1(\OO_{\tilde C_{\ov K}}) = 0$, where $\tilde C_{\ov K}$ is the normalization of $C_{\ov K}$. This implies that the regular curve $C_n$ has genus $0$ for large $n$. 
Thus if some non-smooth point $x$ on $C$ is purely inseparable,
\footnote{The existence of non-smooth purely inseparable points on $C$ can be assumed without loss of generality, by passing, if necessary, to the regular curve $C_L = C \otimes_K L$, where $L$ is the separable closure of $K$.}
the theorem yields an integer $n>0$ such that $x_n$ is a rational point on the genus zero curve $C_n$, hence $C_n \isom \PP^1_K$. 
Now, starting with this explicit description of $C_n$ one can successively obtain explicit descriptions of the regular curves $C_{n-1},C_{n-2},\dots,C$, by applying an algorithm due to Bedoya and St\"ohr \cite{BedSt87}.

The above strategy was applied in \cite[Section~2]{HiSt23} to get a full description of all quasi-elliptic curves.
It has also been applied to obtain a complete classification of all regular geometrically rational plane quartic curves in characteristic two (see \cite{HiSt24}).

\begin{rem}
The sharp bounds $\lam_p(d)$, defined as the smallest $m$ such that for all pairs $(C|K,x)$ with $\de(x) = \frac{p-1}2{\cdot}d$ one has $x_n \in C_n(K)$ for all $n\geq m$, depend on the collection of pairs $(C|K,x)$ considered, 
e.g., restricting the collection of pairs may lead to smaller sharp bounds.
In addition, the sharp bounds depend solely on the properties of the non-smooth points $x$, i.e., they do not depend on the properties of the regular curves $C|K$ at the remaining points $y \neq x$.
Thus one might expect the same sharp bounds $\lam_p(d)$ will hold if we impose suitable \emph{global} restrictions on the pairs $(C|K,x)$, where by \emph{global} we mean conditions not involving the point $x$. 
We check that this is the case for the restrictions $\ov g = h^1(\OO_{\tilde C_{\ov K}}) = 0$ and $\mathrm{pdeg}(K)>1$ (see Remarks~\ref{2026_05_18_18:30} and~\ref{2026_05_18_17:45} below). We expect the same to be true for other global restrictions such as requiring that $x$ be the only non-smooth point on $C$, or that the extended curve $C_{\ov K}$ have a fixed positive geometric genus $\ov g$.

\end{rem}

\subsection{Outline of proof of the theorem}
To compute the sharp bounds $\lam_p(d)$ we construct a class of pairs $(C|K,x)$ whose image points $x_n \in C_n$ are not $K$-rational for suitable $n$.
This yields a lower bound for $\lam_p(d)$ (we have an upper bound by \cite[Theorem~1.1]{HiSt22}), which, after a careful analysis, leads to the actual value of $\lam_p(d)$.
A key ingredient in our approach is the duality between regular curves and function fields
(see \cite[7.4]{EGA2}, or \cite[II.2.5]{Sil09} for the case where the base field is perfect),
which allows us to investigate the problem in the arithmetic setting of function field theory.

\begin{rem}\label{2026_05_18_18:30}
A by-product of our construction is a subclass of pairs $(C|K,x)$ that realize the sharp bounds $\lam_p(d)$.
While explicit descriptions of each point $x$ and its images $x_n$ are given, we have no a priori control over the remaining points on the curve $C$, hence it is hard to describe its global geometry, e.g., compute the genus $g=h^1(\OO_C)=h^1(\OO_{C_{\ov K}})$.
These curves, however, share a common feature: they are all geometrically rational, i.e., each base change $C_{\ov K}$ has geometric genus $\ov g=h^1(\OO_{\tilde C_{\ov K}})=0$.
\end{rem}

\begin{rem}\label{2026_05_18_17:45}
Some of the pairs $(C|K,x)$ we construct to bound $\lam_p(d)$ exist only over imperfect base fields $K$ whose $p$-degree $\mathrm{pdeg}(K) = \log_p [K:K^p]>0$ is at least $2$ (see Remark~\ref{2025_05_21_23:20}).
This implies that the sharp bounds $\lam_p(d)$ hold for pairs $(C|K,x)$ satisfying the restriction $\mathrm{pdeg}(K)>1$. It is not clear what happens when $\mathrm{pdeg}(K)=1$.
\end{rem}



\subsection{Arbitrary non-smooth points}
For points that are not purely inseparable our sharp bounds $\lam_p(d)$ are related to the notion of a separable point.
Given a point $x$ on a regular curve $C|K$, one knows (see \cite[Corollary~2.16]{HiSt22}) that the separable degree $[\ka(x):K]_s$ of the residue field extension $K \su \ka(x)$ divides the integer $\frac 2 {p-1} {\cdot} \de(x)$.

\begin{thm}[see Theorem~\ref{2025_05_22_11:15}]
Let $C|K$ be a regular curve. 
Let $x$ be an (arbitrary) non-smooth point on $C$ with $\de(x) = \frac{p-1}2 {\cdot} [\ka(x):K]_s {\cdot} d > 0$.
Then the image point $x_n \in C_n$ is separable, i.e., $x_n \in C_n(K^{sep})$, for all $n \geq \lam_p (d)$.

Furthermore $\lam_p(d)$ is sharp with respect to this property, that is, there exist a regular curve $C'|K'$ and a non-smooth point $x'$ on $C'$ with $\de(x') = \frac{p-1}2 {\cdot} [\ka(x'):K']_s {\cdot} d > 0$, such that $x'_n \in C_n'$ is not separable for $n=\lam_p(d)-1$.
\end{thm}

The above theorem is a generalization of Theorem~\ref{2025_05_22_11:10}, since a point $x$ is purely inseparable if and only if $[\ka(x):K]_s = 1$, and moreover the image point $x_n$ is rational if and only if it is both separable and purely inseparable.

\subsection*{Acknowledgements}
This research was partly conducted in the framework of the research training group \emph{GRK 2240: Algebro-Geometric Methods in Algebra, Arithmetic and Topology}.
I extend my gratitude to Jo\~ao Pedro dos Santos, Stefan Schr\"oer, Quentin Posva, Michel Brion, and Seoyoung Kim for valuable feedback on previous versions of this manuscript.
I am also grateful to Karl-Otto St\"ohr for helpful conversations and discussions, and to the anonymous referee for comments and careful proof-reading, in particular for saving me from a couple of blunders on more than one occasion.

\section{Preliminaries}

By a \emph{curve} over a field $K$ we mean a proper geometrically integral $K$-scheme of dimension $1$.
A curve is called \emph{regular} (or \emph{normal}) if its local rings are regular local rings.

As follows from \cite[7.4]{EGA2}, there is a one-to-one contravariant correspondence between the regular curves $C$ over a field $K$ and the \emph{one-dimensional separable function fields} $F|K$, given by the assignment $C|K \mapsto F|K = K(C)|K$.
A one-dimensional separable function field is a separable field extension $F|K$ of transcendence degree one, with $K$ algebraically closed in $F$.
Note that if $K \su K'$ is an algebraic extension then $FK'|K' = F \otimes_K K' | K'$ is also a one-dimensional separable function field; the corresponding regular curve over $K'$ is the normalization of the base change $C_{K'} = C \otimes_K K'$.

In the function field setting the (regular) closed points on the curve $C|K$ are called the \emph{primes} of the function field $F|K$,
and they are denoted by the letters $\pp$, $\fq$, etc. 
The local rings $\OO_\pp$ of the primes $\pp$ of $F|K$ are the (discrete) valuation rings of $F|K$.


\begin{defn}\label{2025_05_18_21:20}
The \emph{geometric singularity degree} $\de(\pp)$ of a prime $\pp$ is the $\ov K$-codimension of the semilocal domain $\OO_\pp\otimes_K \ov K \su F \otimes_K \ov K$ in its integral closure.
Equivalently, let $\pi :\tilde C_{\ov K} \to C_{\ov K}$ be the normalization morphism of the extended curve $C_{\ov K} = C \otimes_K \ov K$ over $\ov K$; then $\de(\pp)=\sum_y \de(y)$, where $y$ runs through the points on $C_{\ov K}$ that lie over $\pp \in C$ and $\de(y)$ is the $\ov K$-dimension of the stalk $(\pi_* \OO_{\tilde C_{\ov K}} / \OO_{C_{\ov K}})_y$.
\end{defn}

A prime $\pp$ is called \emph{singular} if $\de(\pp)>0$, i.e., $\pp$ is non-smooth as a point on $C|K$.
Singular primes exist only over base fields $K$ that are imperfect, in particular only if the characteristic $p$ of $K$ is positive.
We assume throughout that $p>0$.

Let $F|K$ be a (one-dimensional separable) function field. The Frobenius pullback $C^{(p)}|K$ of the corresponding regular curve $C|K$ is a curve that may not be regular. 
Passing to its normalization we get a regular curve $C_1|K$ with function field $F_1|K := F^p K |K$.
We call $F_1|K$ the \emph{Frobenius pullback} of $F|K$.
Iterating this process we get iterated Frobenius pullbacks $F_n|K = F^{p^n}K|K$, which are the function fields of the regular curves $C_n|K = \wt{C^{(p^n)}}|K$.
The composite $F_n=F^{p^n} K$ is the only intermediate field of $F|K$ such that the extension $F_n \su F$ is purely inseparable of degree $p^n$.

Given a prime $\pp$, there is an associated discrete valuation $v_\pp$ of $F|K$. Restricting $v_\pp$ to $F_n|K$ and normalizing we obtain a discrete valuation that corresponds to a prime $\pp_n$ of $F_n|K$. We call $\pp_n$ a \emph{restricted prime}, or the \emph{restriction} of $\pp$ to $F_n|K$.

By the Fundamental Equality for function fields, the product between the ramification and inertia indices $e_{\pp|\pp_1}$ and $f_{\pp|\pp_1}$ of each extension $\pp|\pp_1$ equals $[F:F_1]=p$.
If $e_{\pp|\pp_1}=1$, i.e., if the purely inseparable extension $\ka(\pp_1) \su \ka(\pp)$ has degree $p$, we say that $\pp|\pp_1$ is \emph{unramified} or \emph{inertial};
if $e_{\pp|\pp_1}=p$, i.e., $\ka(\pp_1) = \ka(\pp)$, we say that $\pp|\pp_1$ is \emph{ramified}.

By \cite[Proposition~2.9]{HiSt22}, a prime $\pp$ is \emph{separable}, i.e., the residue field extension $K \su \ka(\pp)$ is separable, i.e., $\pp \in C(K^{sep})$, if and only if $\pp$ is non-singular and the extension $\pp|\pp_1$ is ramified.
In particular singular primes do not exist if $K$ is perfect.
Moreover, every prime $\pp$ that is \emph{rational}, i.e., $\ka(\pp)=K$, is also non-singular.

\begin{defn}\label{2025_05_18_23:00}
A prime $\pp$ is called \emph{purely inseparable} (or \emph{non-decomposed}) if it satisfies the following three equivalent conditions (see \cite[Corollary~2.17]{HiSt22})
\begin{enumerate}[\upshape (i)]
    \item there is a unique prime in the extended function field $F \ov K | \ov K$ that lies over $\pp$, i.e., there is a unique point in the extended curve $C_{\ov K}$ that lies over $\pp \in C$;
    \item \label{2025_05_18_23:02}
    the restricted prime $\pp_n$ is rational for some $n$;
    \item \label{2025_05_18_23:03}
    the residue field extension $K \su \ka(\pp)$ is purely inseparable, i.e., $\pp \in C(K^{p^{-\infty}})$.
\end{enumerate}
\end{defn}
Note that if some restricted prime $\pp_n$ is rational then so is each $\pp_m$ for $m\geq n$, since 
$K \su \cdots \su \ka(\pp_2) \su \ka(\pp_{1}) \su \ka(\pp)$.
Clearly $\pp$ is purely inseparable if and only if some (and every) restricted prime $\pp_n$ is purely inseparable.

The local invariants (in particular the singularity degree) of any purely inseparable prime $\pp$ can be calculated from the Bedoya--Stöhr algorithm~\cite{BedSt87}.
We remark that the key condition for the algorithm is item~\ref{2025_05_18_23:02} in Definition~\ref{2025_05_18_23:00} (see \cite[p.\,313]{BedSt87}), although in \cite{BedSt87} one always takes for granted the stronger assumption that $K$ is separably closed.

\begin{thm}[{\cite[Theorem 2.3]{BedSt87}}] \label{2024_01_21_02:30}
Let $\pp$ be a purely inseparable prime in a one-dimensional separable function field $F|K$.
Choose $n>0$ such that the restricted prime $\pp_n$ 
of $F_n|K$ is rational.
Let $z\in F$ be a function such that $\OO_\pp = \OO_{\pp_1}[z]$.
Then
\[ \de(\pp) = p \, \de(\pp_1) + \frac{p-1}2 \cdot v_{\pp_n} (dz^{p^n}). \]
\end{thm}

Here $v_{\pp_n} (dz^{p^n})$ is the order of the exact differential $dz^{p^n}$ of $F_n|K$ at the rational prime $\pp_n$.
Note that a function $z$ with $\OO_\pp = \OO_{\pp_1}[z]$ always exists:
if the extension $\pp|\pp_1$ is ramified, one can take $z$ to be a local parameter at $\pp$;
if $\pp|\pp_1$ is inertial one can select a function $z\in \OO_\pp$ whose residue class $z(\pp)\in \ka(\pp)$ generates the degree $p$ purely inseparable extension $\ka(\pp_1)\su\ka(\pp)$.

The algorithm is applied backwards: starting with $\de(\pp_n)=0$, the theorem yields successively $\de(\pp_{n-1}),\de(\pp_{n-2}),\dots,\de(\pp)$; see Section~\ref{2024_02_13_22:30} for explicit applications.

\begin{cor} \label{2025_06_17_19:00}
Assumptions as in Theorem~\ref{2024_01_21_02:30}.
Then $2\de(\pp) \not \equiv 1 \pmod p$.
\end{cor}

\begin{proof}
Let $t\in F_n$ be a local parameter at the rational prime $\pp_n$.
Write $z^{p^n} \in F_n$ as a Laurent series in $t$ with coefficients in $K$, say $z^{p^n}=\sum_i a_i t^i$.
Then $d z^{p^n}=\sum_i i a_i t^{i-1} dt$,
and since we work in characteristic $p>0$ we conclude $v_{\pp_n}(dz^{p^n}) \not \equiv -1 \pmod p$.
\end{proof}

\section{The sharp bound} \label{2023_07_05_21:20}

We start by fixing some terminology and notation.

\begin{defn}
Let $F|K$ be a one-dimensional separable function field.
Assume that there is a purely inseparable prime $\pp$ in $F|K$.
We call $(F|K,\pp)$ a \emph{pair}.
\end{defn}

By \cite[Proposition~2.5]{HiSt22}, the geometric singularity degree $\de(\pp)$ of a pair $(F|K,\pp)$ is a multiple of $\frac{p-1}2$, that is, $\de(\pp) = \frac{p-1}2 {\cdot}d$ for some integer $d$. 
We say that a positive integer $d$ is \textit{admissible} if there exists a pair $(F|K,\pp)$ such that $\de(\pp) = \frac{p-1}{2} {\cdot} d$.

\begin{prop}\label{2023_07_14_13:25}
    A positive integer $d$ is admissible if and only if $d\not\equiv -1 \pmod p$.
\end{prop}

\begin{proof}
Suppose that $d\not\equiv -1 \pmod p$. 
The function field $F|K=K(x,z)|K$ given by 
\[ z^p = a + x^{d+1}, \quad \text{where $a\in K\setminus K^p$} \] 
has Frobenius pullback $F_1|K = K(x)|K$.
Let $\pp$ be the zero of $x$, i.e., let $\pp$ be the only prime such that $v_\pp(x)>0$.
This prime is purely inseparable, since its restriction $\pp_1$ is the rational $x$-adic valuation of $F_1|K$.
As the residue class $z(\pp)\in \ka(\pp)$, which equals $a^{1/p}$, does not belong to $\ka(\pp_1)=K$ it follows that $\pp$ is inertial over $F_1$ (i.e., $\pp|\pp_1$ is inertial) with residue field $\ka(\pp) = K(z(\pp)) = K(a^{1/p})$.
As the differential $dz^p = (d+1) x^d dx$ of $F_1|K$ has order $d$ at $\pp_1$, we conclude from Theorem~\ref{2024_01_21_02:30} that
\[ \de(\pp) = p\, \de(\pp_1) + \frac{p-1}2 \cdot v_{\pp_1}(dz^p) = \frac{p-1}2 \cdot d.   \]
Thus $d$ is admissible, as desired.
The converse follows from Corollary~\ref{2025_06_17_19:00}.
\end{proof}

\begin{rem}
Let $F|K=K(x,z)|K$ be the function field in the above proof.
Let $\pp'$ be the pole of $x$, i.e., $\pp'$ is the only prime with $v_{\pp'}(x)<0$.
By the Jacobian criterion every prime different from $\pp$ and $\pp'$ is non-singular.
The prime $\pp'$ is also non-singular, actually rational, because $v_{\pp_1'}(z^p)=-(d+1)$ is not divisible by $p$ and so the extension $\pp'|\pp_1'$ is ramified.
Thus $\pp$ is the only singular prime of $F|K$.
Since the rational Frobenius pullback $F_1|K$ has genus $g_1=0$ the extended function field $F\ov K | \ov K$ has genus $\ov g = 0$, hence we conclude from \cite[Formula~2.3]{HiSt22} that the function field $F|K$ has genus $g=\frac{p-1}2 {\cdot} d$.
Note that for $p>2$, $d=1$, this recovers the curve considered in the proof of \cite[Lemma~8.13]{IIL20}.

\end{rem}

For each positive integer $d$ that is admissible
we define
\begin{align*}
    \tau_p(d) &= 
    \begin{cases} 
        \lceil \log_p((p-1)d + 1) \rceil & \text{if $d$ is a sum of consecutive $p$-powers,}  \\
        \lceil \log_p((p-1)d + 1) \rceil - 1 & \text{otherwise,}
    \end{cases}
    \\
    &= 
    \begin{cases} 
        i+1 & \text{if $d=P^i_j$ for some $j \leq i$,}  \\
        i & \text{if $P^{i-1}_0 < d < P_0^i$ and $d \neq P_j^i$ for all $j$ with $j \leq i$,}
    \end{cases}
\end{align*}
where for every pair of integers $i \geq j \geq 0$ the symbol $P_j^i$ denotes 
\[ P_j^i = p^j + \dots + p^i = \sum_{k=j}^i p^k = \frac{p^{i+1} - p^j}{p-1}. \]

\begin{thm}[{\cite[Theorem~2.24]{HiSt22}}] \label{2025_05_18_14:15}
Let $(F|K,\pp)$ be a pair such that $\de(\pp) = \frac{p-1}2 {\cdot} d > 0$.
Then the restricted prime $\pp_n$ is rational for all $n \geq \tau_p (d)$.
\end{thm}

Our goal in this paper is to find a bound similar to $\tau_p(d)$ that is sharp in every characteristic $p>0$.
%
%
In other words, given an admissible integer $d>0$,
we ask for the smallest integer $m$ such that
for every pair $(F|K,\pp)$ with $\de(\pp) = \frac{p-1}2 {\cdot}d$ the restricted prime $\pp_n$ is rational for all $n \geq m$.
We denote the smallest such $m$ by $\lam_p(d)$.
Now Theorem~\ref{2025_05_18_14:15} rephrases as
\begin{equation}\label{2024_05_18_11:40}
    1 \leq \lam_p(d) \leq \tau_p(d).
\end{equation}
The theorem below gives the precise value of $\lam_p(d)$.

\begin{thm}\label{2022_07_24_00:25}
Let $d$ be a positive integer such that $d \not\equiv -1 \pmod p$.
Choose $0\leq r < p-1$ such that $d\equiv r \pmod p$.
Then $\lam_p(d)=\tau_p(d)$ or $\lam_p(d)=\tau_p(d)-1$,
the latter case occurring
if and only if 
there is an integer $i>1$ such that 
$P_0^{i-1} < d < (r+1) P_0^{i-1}$ and
\[ d\neq r P_0^{i-1} + P_j^{i-1} \quad \text{for all $j$ with $0<j\leq i$}.\]
\end{thm}

Note that for $i<j$ we follow the convention that $P_j^i = 0$.
As an immediate consequence we retrieve \cite[Proposition~2.25]{HiSt22}.

\begin{cor}
If a positive integer $d$ satisfies $d\equiv 0 \pmod p$ then the bound $\tau_p(d)$ is sharp for $d$, i.e., $\lam_p(d) = \tau_p(d)$.
In particular, in characteristic $p=2$
the bound $\tau_p(d)$ is always sharp.
\end{cor}

The proof of the theorem requires some preparation. 
In view of \eqref{2024_05_18_11:40}, a first step consists in bounding $\lam_p(d)$ from below.
For this we need explicit examples of pairs $(F|K,\pp)$ with the property that the restricted primes $\pp_n$ are not rational for suitable $n$.
We postpone the discussion of such examples to the next section, and use them here to prove the theorem.

\begin{prop}\label{2025_12_10_12:40}
$ $
\begin{enumerate}[\upshape (i)]
    \item \label{2025_12_10_12:41}
    Let $i\geq j>0$, $\ell\geq 0$ and $r\in\{0,\dots,p-2\}$ be integers. Assume that $d = r P_0^i + \ell P_j^i$ is positive, so in particular $d$ is admissible. Then $i<\lam_p(d)$.
    \item \label{2025_12_10_12:42}
    Let $i > j>0$, $\ell\geq 0$ and $r,r'\in\{0,\dots,p-2\}$ be integers. Assume that     $d = r P_0^i + P_j^i + r' p^j + \ell p^{j+1}$ is positive, so $d$ is admissible. Then $i<\lam_p(d)$.
\end{enumerate}
\end{prop}

The proposition, which follows immediately from Examples~\ref{2022_07_19_19:00} and~\ref{2022_07_19_18:55}, provides lower bounds for $\lam_p(d)$ in specific cases.
We use these specific bounds to derive a general bound for $\lam_p(d)$.


\begin{lem}\label{2022_07_26_22:25}
Let $d$ be a positive integer such that $d\equiv r \pmod p$, where $0 \leq r<p-1$.
Suppose that 
$ r P_0^i + P_1^i - 2p < d$
for some integer $i>0$.
Then $i<\lam_p(d)$.
\end{lem}

\begin{proof}
In view of Proposition~\ref{2025_12_10_12:40},
it is enough to show that $d$ can be written as
\begin{equation}\label{2022_07_20_01:50}
    d= r P_0^{i} + \ell p^{i} \quad \text{for some $\ell\geq 0$,}
\end{equation}
or as
\begin{equation}\label{2022_07_20_01:45}
    d = r P_0^{i} + P_j^{i} + r' p^j + \ell p^{j+1} 
    \,\,\, \text{for some $\ell \geq 0$, $r' \in \{ 0,\dots,p-2 \}$ and $j$ with $0 < j < i$.}
\end{equation}
Note first that 
$-p^i \leq P_1^i - 2p < d - r P_0^i$.
Moreover for each $r' \in \{0,\dots,p-2\}$ and $j$ with $0<j<i$ one has
\[ -p^{j+1} \leq P^i_j -2 p^j - P^i_j - r' p^j \leq P^i_1 -2 p - P^i_j - r' p^j < d - r P_0^i - P_j^i - r' p^j.  \]
Thus, if $d$ does not admit 
a representation as in \eqref{2022_07_20_01:45}, i.e.,
$ d - r P_0^i - P_j^i - r' p^j \not \equiv 0 \pmod{p^{j+1}}$ for each $r' \in \{0,\dots,p-2\}$ and $j$ with $0<j<i$,
then
\[ d - r P_0^{i} \not\equiv (r'+1) p^j \pmod{p^{j+1}} \quad \text{for each $j=0,\dots,i-1$ and $r'=0,\dots,p-2$,} \]
which means that $d - r P_0^{i} \equiv 0 \pmod{p^{i}}$, i.e., 
$d$ admits a representation as in \eqref{2022_07_20_01:50}.
\end{proof}

Given a prime $\pp$, recall (see \cite[pp.\,284-285]{HiSt22}) that the sequence of singularity degrees $\de(\pp_n)$ is non-increasing and converges to zero rapidly in the sense that the partial differences
\[\De_n = \De_n(\pp) := \de(\pp_n) - \de(\pp_{n+1}) \geq 0 \]
satisfy $\De_{n+1} \leq p^{-1} \De_n$. In turn $\de(\pp) = \De_0 + \De_1 + \De_2 + \cdots$.
Moreover each $\De_n$ is a multiple of $\frac{p-1}2$.

\begin{lem} \label{2022_07_24_00:30}
Let $(F|K,\pp)$ be a pair.
Let $n > 0$ be such that $\pp_{n}$ is rational.
Assume that
\[ \De_{i-1}(\pp) = \frac{p-1}{2} \cdot r p^{n-i} \quad \text{for every $i=1,\dots,n$},  \]
where $0 < r < p-1$.
Then the extension $\pp|\pp_n$ is totally inertial, i.e., $[\ka(\pp):\ka(\pp_n)] = [F:F_n] = p^n$.
In particular the degree $\deg(\pp) = [\ka(\pp):K]$ of $\pp$ is equal to $p^{n}$.
\end{lem}

\begin{proof}
We apply descending induction.
The extension $\pp_{n-1}|\pp_n$ is inertial
because $\pp_{n}$ is rational and $\de(\pp_{n-1}) = \De_{n-1}(\pp) > 0$.
Suppose that $\pp_{i}|\pp_n$ is totally inertial for some $i<n$, 
and let us check that $\pp_{i-1}|\pp_i$ is inertial.
If this were false, 
then \cite[Lemma~3.2]{HiSt23} would imply $\de(\pp_{i-1}) - p \, \de(\pp_{i}) \geq \frac{p-1}2 \cdot \deg(\pp_{i})$, i.e., $r \geq p^{n-i}$, a contradiction.
\end{proof}

\begin{rem}\label{2023_09_14_17:15}
According to \cite{HiSt22} the integer $\tau_p(d)$ admits the combinatorial description
\begin{equation}\label{2025_11_10_02:35}
    \tau_p(d) = \max\{ s + \min\{v_p(d_1),\dots,v_p(d_s)\} \},
\end{equation}
where $v_p$ is the $p$-adic valuation of $\Q$ and the maximum is taken over all the partitions
\[ d = d_1 + \dots + d_s \]
of $d$ such that 
\[ d_{i+1} \leq p^{-1} d_i \, \, \text{ for each $i=1,\dots,s-1$}. \]
Every pair $(F|K,\pp)$ with $\de(\pp) = \frac{p-1}2 {\cdot} d$ yields a partition $d = d_1 + \dots + d_s$ of $d$, where $d_i = \frac 2{p-1} {\cdot} \De_{i-1}(\pp)$ and $s$ is the smallest integer such that $\pp_s$ is non-singular.
By the proof of \cite[Proposition~2.22]{HiSt22}, in this case the restricted prime $\pp_n$ is rational for all $n\geq s + \min\{ v_p (d_1), \dots, v_p (d_s) \}$.
Therefore, the bound $m=\tau_p(d)$ in Theorem~\ref{2025_05_18_14:15} is sharp for $d$, i.e., $\lam_p(d) = \tau_p(d)$, if and only if the maximum in \eqref{2025_11_10_02:35} is attained by a partition $d = d_1 + \dots + d_s$ 
that comes from a pair $(F|K,\pp)$ with the property that $\pp_{m-1}$ is non-rational.

\end{rem}

Now we have all the ingredients needed to prove Theorem~\ref{2022_07_24_00:25}.

\begin{proof}[Proof of Theorem~\ref{2022_07_24_00:25}]
Recall that $1 \leq \lam_p(d) \leq \tau_p(d)$.
If $d=P_j^i$ for some $i\geq j\geq0$ then the bound $\tau_p(d)=i+1$ is sharp for $d$, i.e., $\lam_p(d)=\tau_p(d)$, as follows from Proposition~\ref{2025_12_10_12:40}~\ref{2025_12_10_12:41}.
Thus we may assume that there is an integer $i>0$ such that the following holds
\[ P_0^{i-1} < d < P_0^i \quad \text{and} \quad \text{$d\neq P_j^i$ for all $j\leq i$}. \]
We wish to see when the bound $\tau_p(d)=i$ is sharp for $d$.
As the bound is sharp when $\tau_p(d)=1$ we can further assume $i > 1$.
In addition, since the bound is sharp when $(r+1)P_0^{i-1} \leq d < P_0^i$, as
this implies by Lemma~\ref{2022_07_26_22:25} that $i-1<\lam_p(d)$,
we may suppose that
\[ P_0^{i-1} < d < (r+1) P_0^{i-1}, \]
and in turn that $r>0$, $p>2$.
Therefore, we must prove that in this situation
\begin{enumerate}[\upshape (i)]
    \item the bound $\tau_p(d) = i > 1$ is sharp if and only if $d$ can be written as $d= r P_0^{i-1} + P_j^{i-1}$ for some $j$ with $0<j\leq i$;
    \item if the bound $\tau_p(d) = i$ is not sharp then $\lam_p(d) = \tau_p(d) - 1$.
\end{enumerate}

We prove (i). As the if part follows from Proposition~\ref{2025_12_10_12:40}~\ref{2025_12_10_12:41} we just need to show the only if part.
For any partition
$ d = d_1 + \dots + d_s $
of $d$
we have $s+\min\{v_p(d_1),\dots,v_p(d_s)\}=s$,
since $r>0$,
hence in light of Remark~\ref{2023_09_14_17:15} 
the condition that the bound $\tau_p(d)=i$ is sharp means that there exist a partition
$ d = d_1 + \dots + d_i $
of length $i$ and a pair $(F|K,\pp)$ such that
\[ \text{$\pp_i$ is rational} \quad \text{and} \quad \De_{k-1}(\pp) = \frac{p-1}{2} \cdot d_{k} \text{ for all $k = 1,\dots,i$}. \]
It is clear that
\[ d_i \cdot p^{i-k} \leq d_k < p \cdot p^{i-k} \quad \text{for all $k=1,\dots,i$}, \]
since $d_{k+1} \leq p^{-1} d_k$ for each $k$ and $d_1<p^i$ as $(r+1)P_0^{i-1} \leq (p-1)P_0^{i-1} < p^i$.
As $\pp_{i}$ is rational and $\de(\pp_{i-1})=\De_{i-1}(\pp) > 0$ we deduce that $\pp_{i-1}$ has degree $p$ and hence that $p$ divides $d_{k}$ for every $k<i$ (see \cite[Proposition~2.8]{HiSt22}).
Therefore $d_i=r$, so in particular
\[ r p^{i-k} \leq d_k < p \cdot p^{i-k} \quad \text{for all $k=1,\dots,i$.} \]
Choose a positive integer $j\leq i$ that is maximal with respect to the following property
\[ d_k = r p^{i-k} \quad \text{for all $k>i-j$.} \]
If $j=i$ then $d = r P^{i-1}_0 + P_j^{i-1}$ as desired, so we assume $j < i$.
By Lemma~\ref{2022_07_24_00:30}, the prime $\pp_{i-j}$ has degree $p^j$,
which implies that $p^j$ divides $d_k$ for each $k \leq i-j$ (see \cite[Proposition~2.8]{HiSt22}). Now $(r+1) p^j \leq d_{i-j}$, and hence $(r+1) p^{i-k} \leq d_k$ for each $k\leq i-j$, so we conclude that the integer
\[ d - (r P_0^{i-1} + P_j^{i-1}) = \sum_{1\leq k \leq i-j} \big( d_k - (r+1) p^{i-k} \big) \]
is a non-negative multiple of $p^j$. In view of $d<(r+1)P_0^{i-1}$, this means  
$d = r P_0^{i-1} + P_j^{i-1}$.

It remains to prove (ii).
Since $1 \leq \lam_p(d) < \tau_p(d)$ 
we may assume $\tau_p(d)>2$, i.e., $i > 2$. 
As $P^{i-1}_0 < d$ we conclude from Lemma~\ref{2022_07_26_22:25} that $i-2 < \lam_p (d)$, i.e., $\tau_p(d)-2 < \lam_p (d)$.
\end{proof}


A prime $\pp$ is said to be \emph{separable} if the residue field extension $K \su \ka(\pp)$ is separable.
Recall that the separable degree $[\ka(\pp):K]_s$ of $K \su \ka(\pp)$ divides the integer $\frac 2{p-1} \de(\pp)$; see \cite[Corollary~2.16]{HiSt22}.

\begin{thm}\label{2025_05_22_11:15}
Let $F|K$ be a one-dimensional separable function field. Let $\pp$ be a (not necessarily purely inseparable) singular prime with $\de(\pp) = \frac{p-1}2 {\cdot} [\ka(\pp):K]_s {\cdot} d > 0$.
Then the restricted prime $\pp_n$ is separable for all $n \geq \lam_p (d)$.
\end{thm} 

\begin{proof}
Let $\fq$ be a (purely inseparable) prime above $\pp$ in the extended function field $FL|L$, where $L=K^{sep}$.
Its restrictions $\fq_n$ lie above the restrictions $\pp_n$ of $\pp$.
By \cite[Proposition~2.12]{HiSt22} the prime $\fq$ has $\de(\fq) = \de(\pp)/[\ka(\pp):K]_s$ and $\deg(\fq)=\deg(\pp)/[\ka(\pp):K]_s$.
Analogous equalities hold for the restricted primes $\pp_n$ and $\fq_n$. 
Thus $\pp_n$ is separable if and only if $\fq_n$ is rational.
\end{proof}

The proof shows that the integers $\lam_p(d)$ are sharp bounds in the following sense.
Let $d$ be a positive integer that is admissible.
Then $\lam_p(d)$ is the smallest integer $m$ that satisfies the following property: given a one-dimensional separable function field $F|K$ and a prime $\pp$ in $F|K$ with $\de(\pp) = \frac{p-1}2 {\cdot} [\ka(\pp):K]_s {\cdot}d$, the restricted prime $\pp_n$ is separable for all $n \geq m$.

\section{Examples} \label{2024_02_13_22:30}

This section discusses the examples of function fields and purely inseparable primes that were needed in the previous section to bound $\lam_p(d)$ from below.
To compute singularity degrees we systematically use the Bedoya--Stöhr algorithm (see Theorem~\ref{2024_01_21_02:30}).

\begin{example} \label{2022_07_19_19:00}
Let $i\geq j>0$, $r\in\{0,\dots,p-2\}$ and $\ell\geq0$.
We construct a one-dimensional separable function field $F|K$ having a purely inseparable prime $\pp$ such that
\[ \de(\pp) = \frac{p-1}2 \cdot \big( r P_0^i + \ell P_j^{i} \big), \quad \text{$\pp_i$ is non-rational}, \quad \text{$\pp_{i+1}$ is rational}. \]
Let $K$ be an imperfect field containing elements
\begin{equation}\label{2024_08_29_20:15}
    a\in K\setminus K^p \quad \text{and} \quad b\in K \setminus K^p(a^{1/p^{j}}).
\end{equation}
For instance, $K$ can be the function field of the projective plane $\PP^2_k$ over an algebraically closed ground field $k$, or over a finite field $k=\F_q$.
Consider the function field $F|K=K(x,y)|K$ defined by the equation
\[ y^{p^{i+1}} = b^{p^j} + x^{\ell \cdot p^j} (a + x^{r+1}). \]
Set $z:= x^{-\ell} (y^{p^{i-j+1}} - b)$, so that $z^{p^j} = a + x^{r+1}$ and $y^{p^{i-j+1}} = b + x^\ell z$.
The Frobenius pullbacks of $F|K$ are given by
\begin{align*}
    F_n|K &= 
    \begin{cases}
        K(x,z,y^{p^{n}})|K, & \text{if $0\leq n \leq i-j$},\\
        K(x,z^{p^{n-i+j-1}})|K, & \text{if $i-j < n < i+1$}, \\
        K(x)|K, & \text{if $n=i+1$}.
    \end{cases}
\end{align*}
Let $\pp$ be the zero of the function $x$, i.e., $v_\pp(x)>0$.
Its restriction $\pp_{i+1}$ to the rational function field $F_{i+1}|K = K(x)|K$ is the rational $x$-adic valuation.
In particular, $\pp$ is purely inseparable.

We compute the singularity degree $\de(\pp)$, and on the way check that $\pp_i$ is non-rational.
As is clear from $z(\pp)^{p^j} = a \notin K^p=\ka(\pp_{i+1})^p$, for every $n$ with $i-j < n < i+1$ the extension $\pp_n|\pp_{n+1}$ is unramified, or more precisely, the residue field extension $\ka({\pp_{n+1}}) \su \ka({\pp_n})$ is purely inseparable of degree $p$, generated by the residue class $z(\pp)^{p^{n-i+j-1}}$ of the function $z^{p^{n-i+j-1}} \in F_n$.
As the differential $dz^{p^{j}} = (r+1) x^r dx$ of $F_{i+1}|K = K(x)|K$ has order $r$ at $\pp_{i+1}$, from Theorem~\ref{2024_01_21_02:30} we infer
\[ \de(\pp_n) = p \, \de(\pp_{n+1}) + \frac{p-1}2 \cdot v_{\pp_{i+1}} (dz^{p^j}) = p \, \de(\pp_{n+1}) + \frac{p-1}2 \cdot r, \quad (i-j < n < i+1). \]
Similarly, since
\[ y(\pp)^{p^{i-j+1}} = 
\begin{cases} 
    b, & \text{if $\ell>0$},  \\
    b + a^{1/p^j}, & \text{if $\ell=0$},
\end{cases} \]
does not lie in $\ka({\pp_{i-j+1}})^p = K^p(a^{1 / p^{j-1}})$, for every $n\leq i-j$ the extension $\pp_n|\pp_{n+1}$ is unramified and
\[ \de(\pp_n) = p \, \de(\pp_{n+1}) + \frac{p-1}2 \cdot v_{\pp_{i+1}} (dy^{p^{i+1}}) = p\, \de(\pp_{n+1}) + \frac{p-1}2 \cdot (r + \ell p^{j}), \quad (0\leq n \leq i-j), \]
where the last equality is due to the fact that the differential $dy^{p^{i+1}} = x^{\ell p^j}dz^{p^j}$ of $F_{i+1}|K$ has order $\ell p^j + r$ at $\pp_{i+1}$.
Thus the $\De$-invariants of the singular prime $\pp$ are equal to
\[ \De_n(\pp) = \de(\pp_n) - \de(\pp_{n+1}) =
\begin{cases}
    \frac{p-1}2 \cdot (r + \ell) p^{i-n}, & \text{if $0\leq n \leq i-j$,} \\
    \frac{p-1}2 \cdot r p^{i-n}, & \text{if $i-j < n \leq i$,} \\
    0, & \text{if $i<n$.}
\end{cases}
\]
In particular $\de(\pp) =  \frac{p-1}2 \cdot ( r P_0^i + \ell P_j^{i} )$, as required.
Moreover,
\[ \deg(\pp_n) = 
\begin{cases}
    p^{i+1-n}, & \text{if $0\leq n \leq i+1$,} \\
    1, & \text {if $i+1 \leq n$.}
\end{cases}
\]

\end{example}

\begin{example} \label{2022_07_19_18:55}
Let $i>j>0$, $r,r'\in\{0,\dots,p-2\}$ and $\ell\geq0$.
We construct a function field $F|K$ and a purely inseparable prime $\pp$ such that
\[ \de(\pp) =  \frac{p-1}2 \cdot \big( r P_0^i + P_j^i + r' p^j + \ell p^{j+1} \big), \quad \text{$\pp_i$ is non-rational}, \quad \text{$\pp_{i+1}$ is rational}. \]
Take an imperfect field $K$ and consider the function field $ F|K = K(y,u)|K $
given by the relation
\[ z^{p^j + r + 1} = a z^{r + 1} + y^{p^{i-j} (r+1)}, \]
where $z:= u^p - y^{p\ell+r' +1}$ and $a\in K \setminus K^p$.
Let $x:=y^{p^{i-j}}/z$, so that
\[ z^{p^j} = a + x^{r+1}, \quad y^{p^{i-j}} = xz, \quad u^p = z + y^{p\ell + r' + 1}. \]
Then the Frobenius pullbacks take the form
\begin{align*}
    F_n|K &= 
    \begin{cases}
        K(x,z,y,u)|K, & \text{if $n=0$},\\
        K(x,z,y^{p^{n-1}})|K, & \text{if $1\leq n \leq i-j$},\\
        K(x,z^{p^{n-i+j-1}})|K, & \text{if $i-j < n < i+1$}, \\
        K(x)|K, & \text{if $n=i+1$}.
    \end{cases}
\end{align*}
Let $\pp$ be the (purely inseparable) zero of the function $x$.
Its restriction $\pp_{i+1}$ to the rational function field $F_{i+1}|K = K(x)|K$ is a rational prime with local parameter $x$. 
As follows from $z(\pp)^{p^j} = a \notin K^p = \ka(\pp_{i+1})^p$, 
the extension $\pp_{i-j+1}|\pp_{i+1}$ is unramified and 
\[ \de(\pp_n) = p \, \de(\pp_{n+1}) + \frac{p-1}2 \cdot v_{\pp_{i+1}} (dz^{p^j}) = p \, \de(\pp_{n+1}) + \frac{p-1}2 \cdot r, \quad (i-j < n < i+1). \]
Since $y^{p^{i-j}} = xz \in F_{i-j+1}$ is a local parameter at $\pp_{i-j+1}$ 
the extension $\pp_1|\pp_{i-j+1}$ is totally ramified,
that is, for every $n$ with $1\leq n\leq i-j$ 
the extension $\pp_n|\pp_{n+1}$ is ramified and $y^{p^{n-1}}$ is a local parameter at $\pp_n$. 
Using Theorem~\ref{2024_01_21_02:30} 
we deduce
\[ \de(\pp_n) = p \, \de(\pp_{n+1}) + \frac{p-1}2 \cdot v_{\pp_{i+1}} (dy^{p^{i}}) = p \, \de(\pp_{n+1}) + \frac{p-1}2 \cdot (p^j + r), \quad (1\leq n \leq i-j). \]
Now, because $u(\pp) = z(\pp)^{1/p}$ does not lie in $\ka({\pp_1}) = \ka(\pp_{i-j+1}) = K(z(\pp))$ 
the extension $\pp|\pp_1$ is unramified and
\[ \de(\pp) = p \, \de(\pp_1) + \frac{p-1}2 \cdot v_{\pp_{i+1}} (du^{p^{i+1}}) = p \, \de(\pp_1) +  \frac{p-1}2 \cdot \big( (r'+1) p^j + r + \ell p^{j+1} \big). \]
Putting the above together we conclude that the $\De$-invariants of the singular prime $\pp$ are given by
\[ \De_n(\pp) =
\begin{cases}
    \frac{p-1}2 \cdot ((r+1)p^i + r' p^j + \ell p^{j+1}), & \text{if $n=0$,} \\
    \frac{p-1}2 \cdot (r + 1) p^{i-n}, & \text{if $0 < n \leq i-j$,} \\
    \frac{p-1}2 \cdot r p^{i-n}, & \text{if $i-j < n \leq i$,} \\
    0, & \text{if $i<n$.}
\end{cases}
\]
Consequently $\de(\pp) = \frac{p-1}2 \cdot \big( r P_0^i + P_j^i + r' p^j + \ell p^{j+1} \big)$. Furthermore,
\[ \deg(\pp_n) = 
\begin{cases}
    p^{j+1}, & \text{if $n=0$,} \\
    p^j, & \text{if $1 \leq n\leq i-j+1$,} \\
    p^{i+1-n}, & \text{if $i-j+1 \leq n \leq i+1$,} \\
    1, & \text {if $i+1 \leq n$,}
\end{cases}
\]
and in particular the prime $\pp_i$ is not rational, as desired.
\end{example}


Both examples are essential to proving Theorem~\ref{2022_07_24_00:25}.
Indeed, the two examples yield Proposition~\ref{2025_12_10_12:40}, which in turn is used to get the lower bound of $\lam_p(d)$ in Lemma~\ref{2022_07_26_22:25}.

\begin{rem}\label{2025_05_21_23:20}
All we need in Example~\ref{2022_07_19_18:55} for the construction of $F|K$ is a base field $K$ that is imperfect, i.e., we need a base field $K$ whose $p$-degree $\mathrm{pdeg}(K) = \log_p [K:K^p]$ is positive.
By contrast, in Example~\ref{2022_07_19_19:00} we need an imperfect base field $K$ with $p$-degree at least 2; indeed, if $\mathrm{pdeg}(K)=1$ then $K=K^p(a)$ whenever $a\in K\setminus K^p$, i.e., $K^{1/p^n}=K^p(a^{1/p^{n}})$ for all $n\geq 0$ and all $a\in K\setminus K^p$, and therefore the element $b$ in \eqref{2024_08_29_20:15} cannot exist.
The author does not know if there exist function fields $F|K$ with $\mathrm{pdeg}(K)=1$ that satisfy the specified properties in this example.
\end{rem}

\end{document}